\documentclass[12pt]{article}

\usepackage{amssymb,amsmath,amsthm}
\newcommand{\gpr}[2]{ \left\langle #1 \mid #2 \right\rangle}
\newcommand{\gp}[1]{ \left\langle #1 \right\rangle}
\newcommand{\ap}[0]{ \approx}
\newcommand{\gpn}[1]{ \left\langle #1 \right\rangle _F^\#}
\parindent=0in\usepackage[numbers]{natbib}
\newtheorem{lem}{Lemma}
\newtheorem*{thm}{Theorem}

\newtheorem*{defn}{Definition}

\newcommand{\field}[1]{\mathbb{#1}}

\author{Mustafa G\"okhan Benli}

\title{Profinite completion of Grigorchuk's group is not finitely presented}

\begin{document}
\maketitle

\begin{abstract}
In this paper we prove that the profinite completion $\mathcal{\hat G}$ of the Grigorchuk group $\mathcal{G}$ is not finitely presented
as a profinite group. We obtain this result by  showing that $H^2( \mathcal{\hat G},\field{F}_2)$ is infinite dimensional. Also
several results are proven about the finite quotients  $\mathcal{G}/ St_{\mathcal{G}}(n)$ including minimal presentations
and Schur Multipliers.
\end{abstract}

In 1980's R.Grigorchuk constructed groups of automorphisms of rooted trees having extraordinary properties. 
One prototypical example (usually referred as the
\textit{first Grigorchuk Group} and will be  denoted by $\mathcal{G}$ throughout the paper) was the object of study of several researchers in the last 30 years. It
has many interesting properties some of which can be summarized as follows: It is an  example of
a Burnside group i.e. a finitely generated infinite periodic group. It's growth function has intermediate growth
rate and it is the first example of such a group. It is an amenable group which is not elementary amenable. It
is not finitely presented but has a special kind of recursive presentation. It is a just infinite branch group i.e. an 
infinite group whose nontrivial quotients are all finite and its lattice of normal subgroups resembles 
the binary rooted tree.  We refer to  \cite{MR764305},
\cite{MR1786869}, \cite{MR2195454}, \cite{MR2035113}
for generalities about the Grigorchuk group and related topics.

As $\mathcal{G}$ is an interesting group from the point view of discrete groups, its profinite completion $\mathcal{\hat G}$ has
several interesting properties in the class of profinite groups. First of all it coincides with the closure of 
$\mathcal{G}$ in the full automorphism group of the binary rooted tree. It is a just-infinite profinite branch group.
It has finite width (i.e. the lower central factors have bounded rank). Also it has an interesting universal property that
it contains a copy of every countably based pro-2 group. In \cite{MR1776763} it was shown that $\mathcal{\hat G}$
is a  counterexample to a conjecture about just-infinite pro-p groups of finite width.

\hfill

Our main result is about finite presentability of 
$\mathcal{\hat G}$ as a profinite group:

\hfill

\begin{thm}
       
The profinite completion $\mathcal{\hat G}$ of $\mathcal{G}$  is a 3 generated pro-2 group and is not finitely 
presented as a profinite group. 

\end{thm}

The main theorem is proved by showing that the cohomology group  $H^2( \mathcal{\hat G},\field{F}_2)$ is 
infinite dimensional. Naturally, to achieve this goal  we prove various intermediate results related to finite quotients of 
$\mathcal{G}$.
A step-by-step scheme can be summarized as follows:

\begin{enumerate}
 \item Finding presentations for the finite quotients $\mathcal{G}_n=\mathcal{G} / St_{\mathcal{G}}(n)$ (Theorem 1).
 \item Using these presentations to compute Schur Multiplier $H_2(\mathcal{G}_n,\field{Z})=(C_2)^{2n-2}$ and 
$H^2(\mathcal{G}_n,\field{F}_2)=(C_2)^{2n+1}$ (Theorem 2).
\item Using theorem 2 and the  fact that $\mathcal{G}$ is a regular branch group  showing that
$H^2(\mathcal{\hat G}, \field{F}_2)$ is infinite dimensional (Theorem 5).
\end{enumerate}

Also as byproduct, in theorems 3 and  4, we show that relators of the presentations from Theorem 1 are independent and 
find minimal presentations for the finite quotients $\mathcal{G}_n$. \\

The paper is organized as follows:\hfill

In section 1 we give  basic definitions and properties of self-similar groups and specifically of $\mathcal{G}$.
Section 2 is  devoted to the discussion of main results. 
Last section contains the proofs together with intermediate lemmas.\hfill

\hfill

\textbf{Notation:}\\

 $[g,h]$ will always denote the element $g^{-1}h^{-1}gh$. $h^g$ is used for the conjugate $g^{-1}hg$.
 If $G$ is a group and $S$ is a subset, $\gp{S}_G$ denotes the subgroup generated by $S$ and
$\gp{S}_G^\#$ denotes the normal subgroup generated by $S$. $C_n$ denotes the cyclic group of order $n$ and $\field{F}_p$ denotes the finite field of order $p$.
 $G^n$ and $G^\infty$ denote the $n$ fold and infinite direct product of $G$ with itself respectively. $\hat G$
and $\hat G_p$ denote profinite and pro-p completions respectively.

\section{Preliminaries}

\subsection*{Automorphisms of rooted trees}

Let $X=\{0,\ldots,d-1\}$ be an alphabet of $d$ elements. Then the set $X^*$ of finite sequences over $X$ has 
the structure of a regular $d$-ary rooted tree. At the root one has the empty string denoted by $\varnothing$, and each
word $w \in X^*$ has $d$ children $\{wx \mid x \in X\}$. We will make no distinction between the set $X^*$ and the 
tree it describes. The ambient object is the group of graph automorphisms 
of $X^*$ denoted by $Aut(X^*)$. These are bijections on $X^*$ which preserve incidence of vertices (i.e. 
prefixes in $X^*$). By the level of a vertex $w$ we mean its length as a word (equivalently its distance to
the root ), we denote the level by $|w|$. It is easy to see that any such automorphism fixes the root and permutes vertices in the same 
level.

\begin{defn}
 Given a subgroup $G$ of $Aut(X^*)$, for each $n\geq 1$ we have a normal subgroup 
$$St_G(n)=\{g\in G \mid g(w)=w \quad \text{for all} \quad w \in X^* \text{ with } |w|=n\} $$

called the n-th level stabilizer of $G$.
\end{defn}

Given a group $G$ of tree automorphisms we will denote $G/St_G(n)$ by $G_n$.\hfill

\hfill

Since each level of the tree has finitely many elements and automorphisms do not change 
levels, it follows that the index $[G:St_G(n) ] $ is always finite. Also since $\cap St_G(n)$
is trivial, any subgroup of $Aut(X^*)$ is residually finite.\hfill

\hfill

The d-ary rooted tree $X^*$ is a self-similar geometric object. The subtree $wX^*$ hanging down
at a vertex $w \in X^*$ (i.e. words starting with $w$) is canonically isomorphic to the whole tree $X^*$
via the isomorphism
$$
\begin{array}{cccl}
 \phi_w :& X^* & \longrightarrow & wX^* \\ 
         & v   & \mapsto         & wv  
\end{array}
$$

This self-similarity also reflects upon the automorphism group:\\
Given an automorphism $f \in Aut(X^*)$ and a vertex $v\in X^*$ we have another automorphism denoted by $f_v$
(called \textit{the section of $f$ at $v$}) which is uniquely determined by the equation 

$$f(vw)=f(v)f_v(w) \quad \text{for all} \quad   v,w \in X^*$$

(Or equivalently $f_v(w)=\phi^{-1}_ {f(v)}f(\phi_v(w))$).

\begin{defn} A subgroup $G$ of $Aut(X^*)$ is called self-similar if it is closed under taking sections of
its elements i.e. for all $f \in G$ and for all $v\in X^*$ we have $f_v  \in G$.
\end{defn}

If $G$ is self-similar  we have an embedding of $G$ into the semi-direct product

$$ 
\begin{array}{ccccc}
G & \longrightarrow & (G \times \ldots \times G) & \rtimes &S_{d} \\
f & \mapsto         & (f_0,\ldots,f_{d-1}) & &\sigma_f 
\end{array}
$$

where $S_d$ is the symmetric group on $d$ letters and $\sigma_f \in S_d$ is the permutation determined by $f$ on 
the first level of the tree. $f_0,\ldots,f_{d-1}$ determine how $f$ acts on the first level subtrees and
$\sigma_f$ determines how these subtrees are permuted. This semi-direct product is called a 
\textit{permutational wreath product} and is usually denoted by $G \wr S_d$.\hfill

\hfill

An easy way to create a self-similar group is to start with a set of symbols $\{f_0,\dots,f_k\}$
and look at the system

$$
\begin{array}{cc}
 f_0=(f_{00},\ldots,f_{0d-1}) & \sigma_0 \\
 f_1=(f_{10},\ldots,f_{1d-1}) & \sigma_1 \\
 \ldots                       & \dots   \\
f_k=(f_{k0},\ldots,f_{kd-1}) & \sigma_k \\
\end{array}
$$

where $f_{ij}\in \{f_0,\dots,f_k\}$ and $\sigma_i \in S_d$.\hfill

\hfill

Such a system (usually referred to as a \textit{wreath recursion}) defines a unique set of automorphisms of $Aut(X^*)$ denoted by $f_0,\dots,f_k$. Then 
we have the subgroup $G=<f_1,\ldots,f_k>_{Aut(X^*)}$ which is obviously a self-similar group since each 
section of each generator is a again a generator. This construction is a source of many interesting 
groups and most of the well studied self-similar groups are defined by a wreath recursion.

In fact a self-similar group defined in this way belongs to a smaller class of self-similar groups 
called the \textit{groups generated by finite automata}. For more on groups generated by automata 
see \cite{MR1841755}.\hfill

\hfill

If $G$ is a self-similar group acting on the d-ary tree,  we have a monomorphism 

$$
\begin{array}{cccc}
 \varphi: & St_G(1) & \longrightarrow & G \times \ldots \times G \\
          &   g       &  \mapsto        & (g_0, \ldots  ,g_{d-1}) 
\end{array}
$$

\begin{defn}
Let $G$ be a level-transitive self-similar group (i.e. it acts transitively on the levels of the tree).
 $G$ is called regular branch over a finite index subgroup $K$ if
$$K \times \ldots \times K \leq \varphi(K) $$

\end{defn}

\subsection*{The Grigorchuk Group}
\begin{defn} Let $X=\{0,1\}$ so that $X^*$ is the binary rooted tree. Consider the following automorphisms
given by the wreath recursion:

$$
\begin{array}{ccc}
a= & (1,1) & \sigma \\
b= & (a,c) &         \\
c= & (a,d) &         \\
d= & (1,b) &         \\
\end{array}
$$

The subgroup $\mathcal{G}$ they generate is called the Grigorchuk group. ($\sigma$ denotes the nontrivial element in $S_2$).
\end{defn}

We will list some well known properties of $\mathcal{G}$ which will be used throughout the paper (For proofs see \cite{MR1786869},\cite{MR2035113}):

\begin{itemize}
 \item $\mathcal{G}$ is level transitive.
 \item $St_\mathcal{G}(1)=\gp{b,c,d,b^a,c^a,d^a}_{\mathcal{G}}$ and we have a monomorphism
$$
\begin{array}{cccc}
 \varphi:& St_\mathcal{G}(1) & \longrightarrow & \mathcal{G}\times \mathcal{G}\\
& b & \mapsto         & (a,c)\\
& c & \mapsto         & (a,d)\\
& d & \mapsto         & (1,b)\\
& b^a & \mapsto         & (c,a)\\
& c^a& \mapsto         & (d,a)\\
& d^a & \mapsto         & (b,1)\\
\end{array}
 $$
\item In $\mathcal{G}$ the relations $a^2=b^2=c^2=d^2=bcd=(ad)^4=1$ hold.
\item $\mathcal{G}$ is an infinite $2$ group.
\item $\mathcal{G}$ is a regular branch group over the subgroup  $K=\gp{(ab)^2}_\mathcal{G}^\#$.
\item $\mathcal{G}$ is just-infinite.
\item $St_\mathcal{G}(3)\leq K$ .
\item $\mathcal{G}_1\cong C_2$, $\mathcal{G}_2\cong C_2 \wr C_2$, $\mathcal{G}_3\cong C_2\wr C_2 \wr C_2$ and $|\mathcal{G}_n|=2^{5. 2^{n-3}+2}$ for $n\geq 3$.
\item For $n\geq4$ the kernels of the quotient maps $q_n : \mathcal{G}_n\longrightarrow \mathcal{G}_{n-1}$ are elementary 
abelian $2$-groups and satisfy

$$|Ker(q_n)|=2^{5.2^{n-4}} $$ 
\end{itemize}

\section{Main Theorems}

\subsection*{Presentations, Schur Multipliers and independence of relators}
Immediately after discovering his group, Grigorchuk proved that it is not finitely presented.
In 1985 a recursive presentation was found by I.Lysenok. He proved that $\mathcal{G}$ has
the following special recursive presentation:

\begin{thm}{(Lysenok \cite{MR819415})} Grigorchuk group has the presentation
$$\gp{a,b,c,d \mid a^2,b^2,c^2,d^2,bcd,\sigma^i((ad)^4),\sigma^i((adacac)^4),i\geq0}, $$

where $\sigma$ is the substitution
$$
\sigma = \left\{
\begin{array}{ccc}
   a & \mapsto & aca \\
   b & \mapsto  & d \\
   c & \mapsto & b \\
    d & \mapsto & c
  \end{array}
\right.$$
\end{thm}

In \cite{MR1676626} Grigorchuk gave a systematic way of finding similar presentations in the 
general case and suggested the name \textit{L-presentations}. This scheme was later used in
\cite{MR2352271},\cite{MR1902367} to find similar presentations for other self similar groups including
iterated monodromy groups.  Also in \cite{MR2009317} Bartholdi showed that every contracting self-similar branch group has 
such a presentation and used his theorem to find presentations for many other well 
known self-similar groups.\hfill

Roughly, a finite $L$-presentation is a generalization of a finite presentation in which one can 
obtain all relations by applying finitely many free group endomorphisms to finitely many initial 
relators. A precise definition is as follows:

\begin{defn}
  An L-presentation (or endomorphic presentation) is an expression \hfill
$$ \gp{X\mid Q\mid R \mid \Phi}$$

where $X$ is a set, $Q,  R $ are subsets of the free group $F(X)$ on the set $X$ and $\Phi$ is a 
set of free group endomorphisms. \hfill

\hfill

It defines the  group

$$G=F(X)/N $$ where 
$$N=\gp{Q, \bigcup_{\phi \in \Phi^*}\phi(R)}_F^\# $$

It is called a finite L-presentation if $X,Q,R,\Phi$ are all finite.
\end{defn}

For more on $L$-presentations see \cite{MR2009317}.\hfill

\hfill

Lysenok's presentation was later used by Grigorchuk \cite{MR1616436} to embed $\mathcal{G}$ into a finitely 
presented amenable group $\tilde{\mathcal{G}}$ which is an HNN-extension of $\mathcal{G}$. Since  $\tilde{\mathcal{G}}$ 
contains $\mathcal{G}$ it is not elementary amenable and hence amenable and elementary amenable groups
do not coincide even in the class of finitely presented groups.\hfill

\hfill

Our first theorem gives  similar presentations for the finite quotients $\mathcal{G}_n$ and will be proved 
in section $3.1$.\hfill

\hfill

\textbf{Theorem 1}\textit{ For $n \geq 3 $ we have 
$$\mathcal{G}_n=\gpr{a,b,c,d}{a^2,b^2,c^2,d^2,bcd,u_0,\ldots,u_{n-3},v_0,\ldots,v_{n-4},w_n,t_n}$$}
\textit{where 
$$u_i=\sigma^i((ad)^4),\;v_i=\sigma^i((adacac)^4),\;w_n=\sigma^{n-3}((ac)^4),\; t_n=\sigma
^{n-3}((abac)^4)$$}
\textit{and $\sigma$ is the substitution given by }

$$
\sigma = \left\{
\begin{array}{ccc}
   a & \mapsto & aca \\
   b & \mapsto  & d \\
   c & \mapsto & b \\
    d & \mapsto & c
  \end{array}
\right.$$

As it can bee seen readily, as $n$ grows these presentations approach to the Lysenok's presentation.\hfill

\hfill

 Recall that given a group $G$, the \textit{Schur Multiplier} of $G$ (denoted by $M(G)$) is the second integral homology group 
$H_2(G,\field{Z})$. For finite groups we have the isomorphism 
$H_2(G,\field{Z})\cong H^2(G,\field{C}^*)$. If $G$ is given by a presentation $F/R\cong G$ 
where $F$ is a free group, the Hopf's formula (obtained first by Schur for finite groups and generalized to infinite
groups by Hopf) gives 

$$M(G)\cong R\cap F' / [R,F] $$

Hence the abelian group $ R\cap F' / [R,F]$ is independent of the presentation of the group. If the given
presentation is finite (i.e. $F$ has finite rank and $R$ is the normal closure of finitely many elements 
$\{r_1,\ldots,r_m\}$ in $F$), then it is easy to see that the abelian group $R/[R,F]$ is generated by the images 
of $\{r_1,\ldots,r_m\}$ and hence its subgroup $R\cap F' / [R,F]$ is a finitely generated abelian group. Therefore
the Schur multiplier of a finitely presented group is necessarily finitely generated. The converse of this is 
not true. Baumslag in \cite{MR0297845}  gave an example of a non finitely presented group with trivial multiplicator.
For generalities about Schur multipliers of groups see \cite{MR1200015}.\hfill

\hfill

The computation of the Schur multiplier of $\mathcal{G}$ was done by Grigorchuk:

\begin{thm}{(Grigorchuk, \cite{MR1676626})} $ M(\mathcal{G}) \cong (C_2)^\infty$
\end{thm}

The proof of this theorem also shows that there are no redundant relators in Lysenok's presentation:

\begin{thm}{(Grigorchuk, \cite{MR1676626})}
 The relators in the Lysenok presentation are independent, i.e. none of the relators is a consequence 
of the others.
\end{thm}

Recently it was shown in \cite{corn} that there are finitely generated  infinitely presented solvable groups which 
do not have independent set of relators. \\

Our second theorem is the computation of Schur multipliers for the finite groups $\mathcal{G}_n$. It relies on 
theorem 1 and and is similar to Grigorchuk's computations done in \cite{MR1676626} with small modifications. The proof is 
presented in section 3.2. It was indicated to us by L.Bartholdi that a shorter proof for Theorem 2 could be given using
ideas in \cite{MR2665774}.\hfill

\hfill

\textbf{Theorem 2.} \textit{$M(\mathcal{G}_n)\cong C_2^{2n-2}$}\hfill

\hfill

Similarly we also prove that the relators of the presentations of Theorem 1 are independent, again the proof 
 will be presented in section 3.2.\hfill

\hfill

\textbf{Theorem 3.} \textit{The relators in the presentations of theorem 1 are independent.}

\subsection*{Minimality of presentations and deficiency of groups}

We begin this subsection with some definitions:

\begin{defn}If $G$ is a group let $d(G)$ denote the minimal number of generators of $G$.
\end{defn}

\begin{defn}
 The minimal number $m$ such that $G$ has a presentation  $$G=\gp{x_1,\ldots,x_t \mid r_1,\ldots,r_m} $$
(a presentation with $m$ relators)
will be denoted by $r(G)$.
\end{defn}

\begin{defn}
 The deficiency of $G$ (denoted by $def(G)$) is defined to be the  minimal $m-t$ such that $G$ has a presentation with $t$
generators and $m$ relators.
\end{defn}

\begin{defn}
 A presentation $G=\gp{x_1,\ldots,x_t \mid r_1,\ldots,r_m}$ is called minimal if $t=d(G)$ and 
$m=r(G)$.
\end{defn}

The following question is open (See \cite{MR0457538}):\hfill

\hfill

\textit{Do  finite groups have  minimal presentations?} \hfill

\hfill

A stronger question is the following:\hfill

\hfill

\textit{Does a finite group  have a presentation realizing its deficiency with $d(G)$ number of generators?}
\hfill

\hfill

Clearly an affirmative answer to the second question gives an affirmative answer to the first.
 Lubotzky \cite{MR1848964} gave affirmative answer to the analogous question in the category of 
profinite groups.
It was proven by Rapaport \cite{MR0308277} that the second question has affirmative answer for 
nilpotent groups. Therefore the groups $\mathcal{G}_n$ have minimal presentations. Our next results exhibits
such a minimal presentation for $\mathcal{G}_n$ which is obtained from the presentations of Theorem 1
by a simple Tietze transformation.  The proof relies on the following inequality for the deficiency:

\hfill

Given a presentation $G=\gp{x_1,\ldots,x_t \mid r_1,\ldots,r_m}=F/R$ of a finite group
$G$,  we have the quotient map 
$$\phi :  R/[R,F] \rightarrow R / (R\cap F') $$ whose kernel is the Schur multiplier $M(G)$. 
But $R/(R \cap F') \cong RF'/ F'$ which is free abelian of  rank $t$ because $R$ has finite index in $F$.
Hence 
$$d(M(G))=d(R/[F,R])-t \leq m-t $$ and since $M(G)$ does not depend on the presentation, the following inequality holds

\begin{equation}
  0 \leq d(M(G)) \leq def(G) \leq m-t
\label{defi}
\end{equation}

\textbf{Theorem 4.} \textit{For $n \geq 3 $ we have 
$$\mathcal{G}_n=\gpr{a,b,c}{a^2,b^2,c^2,(bc)^2 ,u_0,\ldots,u_{n-3},v_0,\ldots,v_{n-4},w_n,t_n}$$}
\textit{where 
$$u_i=\sigma^i((abc)^4),\; v_i=\sigma^i((abcacac)^4),\; w_n=\sigma^{n-3}((ac)^4),\; t_n=\sigma
^{n-3}((abac)^4)$$}
\textit{and $\sigma$ is the substitution given by}

$$
\sigma = \left\{
\begin{array}{ccc}
   a & \mapsto & aca \\
   b & \mapsto  & bc \\
   c & \mapsto & b \\

  \end{array}
\right.$$

\textit{and this presentation is minimal and realizes the deficiency  $def(\mathcal{G}_n)=2n-2$.}

\begin{proof}
 The presentations found in theorem 1 contain the relators $d=bc$. Hence applying Tietze transformations we
get the asserted presentations. By theorem 2 we have $d(M(\mathcal{G}_n))=2n-2$. Using equation (\ref{defi}) and 
counting generators and relators in the above presentation we get $ 2n-2\leq def(\mathcal{G}_n) \leq 2n-2 $.\hfill

\hfill

We have $\mathcal{G}_3^{ab}\cong (C_2)^3$ and $\mathcal{G}_n$ maps onto $\mathcal{G}_3$. Also $\mathcal{G}^{ab}\cong (C_2)^3 $ and $\mathcal{G}$  maps onto $\mathcal{G}_n$.
These show that $\mathcal{G}_n^{ab}\cong (C_2)^3$ and  $d(\mathcal{G}_n)=3$. Hence the above presentation realizes the deficiency 
with minimal number of generators. Therefore it is necessarily minimal.
\end{proof}

\subsection*{Profinite completion of the Grigorchuk Group}

The full automorphism group $Aut(X^*)$ is a profinite group. It is the inverse limit of the system
$$\{Aut(X^*_n) \mid n\geq 1\} $$ 

where $Aut(X^*_n)$ denotes the automorphism group of the finite tree $X^*_n$ consisting of the first $n$ levels and the 
maps $$\phi_n : Aut(X^*_n) \longrightarrow Aut(X^*_{n-1}) $$ are given by restriction.\hfill

\hfill

Given $G \leq Aut(X^*)$ one can talk about 3 groups $\hat G,\; \hat G_p$ and $\bar G$ where the 
last one denotes the closure of $G$ in $Aut(X^*)$. Since the Grigorchuk group $\mathcal{G}$ is a $2$-group it follows
that $ \mathcal{\hat G} =  \mathcal{ \hat G}_2$. It is also true that $ \mathcal{\bar G}$ coincides with these groups because of the following:

\begin{defn}
 A subgroup $G\leq Aut(X^*)$ is said to have the congruence property if every finite index subgroup 
of $G$ contains the subgroup $St_{G}(n)$ for some $n$.
\end{defn}

\begin{thm}{(See \cite{MR1765119})} $\mathcal{G}$ has the congruence property.
\end{thm}
Now it follows that $ \mathcal{\hat G} \cong  \mathcal{\bar G}$ because the congruence property shows that

$$\{St_\mathcal{G}(n)\mid  n\geq 1\} $$ is a neighborhood basis of the identity in $\mathcal{G}$.\\

The congruence property also shows that $\mathcal{\hat G}$ is the inverse limit of the inverse system 
$\{\mathcal{G}_n\:,\: n \geq 1 \}$.\\

Recall that  a profinite group  $G$ is finitely presented (as a profinite group) if 
there is an exact sequence

$$1\longrightarrow R \longrightarrow \hat F \longrightarrow G \longrightarrow 1 $$

where $F$ is a free group with finite rank and $R$ is the closed normal subgroup of $\hat F$ generated 
by some $\{r_1,\ldots,r_m\}\subset R$. \\

Clearly if $G\cong F/R$ is a presentation of  a discrete group $G$, then  $\hat G \cong \hat F / \bar R$ is a
profinite presentation for $\hat G$ where $\bar R$ denotes the closure of $R$ in $\hat F$. Therefore profinite completions of finitely presented groups
are necessarily finitely presented. It was indicated us by Lubotzky (private communication) that the converse of 
this statement is not true. That is there are finitely generated residually finite groups $G$ and $H$ with $G$ finitely
presented and $H$ not finitely presented and $\hat G$ is isomorphic to $\hat H$.
Therefore one can ask  whether the profinite completion of the Grigorchuk group is finitely 
presented or not. Our last theorem shows that it indeed is not finitely presented. It relies on the following
well known fact:

\begin{thm}{(See \cite{MR1691054} page 242)}
A finitely generated pro-$p$ group $G$ is finitely presented if and only if $H^2(G,\field{F}_p)$ is 
finite.
 
\end{thm}

\textbf{Theorem 5} \textit{We have $H^2( \mathcal{\hat G},\field{F}_2)\cong (C_2)^\infty$ and hence $ \mathcal{\hat G}$ is not finitely presented as a
profinite group.}\hfill

\hfill

The proof of theorem 5 is presented in section 3.3.

\section{Proofs of Theorems}
\subsection{Finding Presentations for $\mathcal{G}_n$}

This section is devoted to the proof of theorem 1. \hfill

\hfill

Let  $\Gamma=\gpr{a,b,c,d}{a^2,b^2,c^2,d^2,bcd,(ad)^4}$. Let us denote by  $\pi: \Gamma \longrightarrow
\mathcal{G}$ the canonical surjection. Consider the subgroup 
$\Xi=\gp{b,c,d,b^a,c^a,d^a}_{\Gamma}$ which is the lift of the first level stabilizer $St_\mathcal{G}(1)$ to $\Gamma$.\hfill

\hfill

We have a homomorphism 

$$
\begin{array}{cccc}
 \bar \varphi:& \Xi & \longrightarrow & \Gamma \times \Gamma\\
& b & \mapsto         & (a,c)\\
& c & \mapsto         & (a,d)\\
& d & \mapsto         & (1,b)\\
& b^a & \mapsto         & (c,a)\\
& c^a& \mapsto         & (d,a)\\
& d^a & \mapsto         & (b,1)\\
\end{array}
 $$

which is analogous to $\varphi : St_\mathcal{G}(1) \longrightarrow \mathcal{G} \times \mathcal{G}$.\hfill

\hfill

(The fact that $\bar\varphi$ is well defined can be checked by first finding a presentation for 
$\Xi$ using Reidemeister-Schreier process and checking that it maps relators to relators.)

Given $w \in \Xi$ let us write $\bar\varphi (w)=(w_0,w_1)$ which is consistent with  the section notation of
tree automorphisms.\hfill

\hfill

Recall the substitution $\sigma$ from theorem 1  given by 
$$
\sigma = \left\{
\begin{array}{ccc}
   a & \mapsto & aca \\
   b & \mapsto  & d \\
   c & \mapsto & b \\
    d & \mapsto & c
  \end{array}
\right.$$

It is easy to check that given $w \in \Xi$ one has $$\bar \varphi (\sigma(w))=(v,w) $$

where $v \in \gp{a,d}_{\Gamma}\cong D_8$. Therefore since all $u_i,v_i,w_i,t_i$ are $4$-th powers and $D_8$
has exponent $4$, we have the following equalities:
\begin{eqnarray}
\bar \varphi(u_i)=(1,u_{i-1})  \nonumber \\
\bar \varphi(v_i)=(1,v_{i-1}) \nonumber \\
\bar \varphi(w_i)=(1,w_{i-1}) \nonumber \\
\bar \varphi(t_i)=(1,t_{i-1}) 
\label{incik}
\end{eqnarray}

Let $ \Omega=Ker(\pi)$ so that $\mathcal{G}=\Gamma / \Omega$. It is known (for example see \cite{MR1786869}) 
that $\Omega$ is a strictly increasing union  
$\Omega=\bigcup_{n}\Omega_n$ , $\Omega_n \subset \Omega_{n+1}$ (This clearly shows that $\mathcal{G}$ is not finitely presented).  
The subgroups  $\Omega_n$ can be defined recursively as follows:
$$\Omega_1 = Ker(\bar \varphi)\quad \text{and} \quad  \Omega_n=\{w \in \Xi \mid w_0,w_1 \in \Omega_{n-1} \} $$ 

It is known that  $\Omega_n=\gp{u_1,\ldots,u_n, v_0,\ldots.v_{n-1}}_{\Gamma} ^\#$ (see \cite{MR1616436}).\hfill

\hfill

The subgroups $\Omega_n$ are related to the "branch algorithm" which solves 
the word problem in $\mathcal{G}$ (See \cite{MR2195454}). Roughly speaking $\Omega_n$ consists of elements for 
which the algorithm stops after  $n$ steps. \hfill

\hfill

Similarly we have subgroups $\Upsilon_n$ of $\Gamma$ such that $\mathcal{G}_n=\mathcal{G}/St_\mathcal{G}(n)=\Gamma/\Upsilon_n$ where 
$\Upsilon_{n+1}\subset \Upsilon_n$ and $\bigcap_n \Upsilon_n=\Omega$. Hence $St_\mathcal{G}(n)=\Upsilon_n / \Omega$.
A recursive definition for $\Upsilon_n$ is:
$$\Upsilon_1 =\Xi $$ and
$$\Upsilon_n =\{w \in \Xi \mid  w_0,w_1 \in \Upsilon_{n-1} \} $$

We will prove  theorem 1 by showing that for 
$n \geq 3$ we have $$\Upsilon_n=\gp{u_1,\ldots,u_{n-3},v_0,\ldots,v_{n-4},w_n,t_ n}_{\Gamma}^\# $$

This will be done by induction on $n$ and the case $n=3$ follows from the following 3 lemmas:

\begin{lem} We have 
 \begin{equation}
\mathcal{G}_3 \cong (C_2 \wr C_2) \wr C_2= \gpr{x,y,z}{x^2,y^2,z^2,[x,x^y],[y,y^z],[x,x^z],[x,y^z],[y,x^z]}
\label{wrpres}
\end{equation}
where  $x=ada,y=c,z=a $
\end{lem}

\begin{proof}
 Direct inspection of the action of $a,c,ada$ on the tree consisting of the first 3 levels
(See \cite{MR1786869} page 226).
\end{proof}

\begin{lem}
 Presentation (\ref{wrpres}) is equivalent to
 \begin{equation}
 \mathcal{G}_3=\gpr{a,b,c,d}{a^2,b^2,c^2,d^2,bcd,(ad)^4,(ac)^4,(adac)^4}
\label{presg3}
\end{equation}
\end{lem}
\begin{proof}
 Follows from the following equations and applying the Tietze transformations to (\ref{wrpres}).

$$\begin{array}{l}

[x,x^y]=[ada,cadac]=(adac)^4 \\
 
[y,y^z]=[c,aca]=(ca)^4 \\

[x,x^z]=[ada,d]=(ad)^4 \\

[x,y^z]=[ada,aca]  \\

[y,x^z]=[c,d]

\end{array}$$

\end{proof}

\begin{lem}
Presentation (\ref{presg3}) is equivalent to 

\begin{equation}
 \mathcal{G}_3=\gpr{a,b,c,d}{a^2,b^2,c^2,d^2,bcd,(ad)^4,(ac)^4,(abac)^4}
\end{equation}

and hence $\Upsilon_3=\gp{w_3,t_3}_{\Gamma}^\#$.
\end{lem}
\begin{proof}
 Follows from the following equalities:
$$(adac)^4=(adacadac)^2=(adcacacdac)^2=(abacabac)^2=(abac)^4 $$

where in  step 3 we use the equality $aca=cacac $.
\end{proof}

Recall that $\mathcal{G}$ is regular branch over the subgroup $K=\gp{(ab)^2}_{\mathcal{G}} ^\#$.

\begin{lem}
 $St_\mathcal{G}(3) \leq K$ and hence $$\varphi (St_\mathcal{G}(n))=St_\mathcal{G}(n-1) \times St_\mathcal{G}(n-1) $$

for $n \geq 4$ , therefore we have

$$ \bar \varphi(\Upsilon_n)=\Upsilon_{n-1} \times \Upsilon_{n-1} $$

for $n \geq 4$.
\end{lem}

\begin{proof}
 The fact that $St_\mathcal{G}(3) \leq K$ is proven in \cite{MR1786869} (page 230). Therefore, since $\mathcal{G}$ is a regular 
branch group over $K$ (i.e $K \times K \leq \psi(K) $) we get the remaining equalities.
\end{proof}

\textbf{Proof of theorem 1.} We have to show that for $n \geq 3$ we have $$\Upsilon_n=\gp{u_1,\ldots,u_{n-3},v_0,\ldots,v_{n-4},w_n,t_ n}_{\Gamma}^\# $$

Now induction on $n$, equations (\ref{incik}) and the fact  $\bar \varphi (\Upsilon_n)=\Upsilon_{n-1} \times \Upsilon_{n-1}$ 
show that
$$\Upsilon_n= Ker(\bar \varphi)\gp{u_2,\ldots,u_{n-3},v_1,\ldots,v_{n-4},w_n,t_ n}_{\Gamma}^\# $$

but $Ker(\bar \varphi)=\Omega_1=\gp{u_1,v_0}_{\Gamma} ^\#$ from which we obtain

$$\Upsilon_n=\gp{u_1,\ldots,u_{n-3},v_0,\ldots,v_{n-4},w_n,t_ n}_{\Gamma}^\# $$

\subsection{Computation of Schur Multiplier of $\mathcal{G}_n$}

This section is devoted to the proofs of theorems 2 and 3.  The ideas are analogous to 
\cite{MR1676626} with slight modifications where needed.\hfill

\hfill

Let $F$ be the free group on $\{a,b,c,d\}$
and let $$K_n = \gpn{a^2,b^2,c^2,d^2,bcd,u_0,\ldots,u_{n-3},v_0,\ldots,v_{n-4},w_n,t_n} $$

so that by theorem 1 we have $F/K_n \cong \mathcal{G}_n$. As mentioned before, the Schur multiplier
can be calculated using Hopf's formula by:

$$M(\mathcal{G}_n )\cong K_n \cap F' /[K_n,F] $$

We have the following basic fact which will be used in the remainder:

\begin{lem}
 We have the following inclusions:
$$K_{n+1}\subset K_n$$ $$\sigma(K_n)\subset K_n$$  $$\sigma([K_n,F])\subset [K_n,F]$$
\end{lem}

\begin{proof}
 The first inclusion follows from the fact that $\mathcal{G}_{n+1}$ maps onto $\mathcal{G}_n$. The images of generators of $K_n$ 
clearly lie in $K_{n+1}$ and hence the second inclusion follows from the first one. Finally the third inclusion follows
directly from the second.

\end{proof}

For computational reasons we need to change the relators in presentation of theorem 1 slightly to the ones
given in the next lemma. The rationale behind this will be apparent when we will do computations modulo the
subgroup $[K_n,F]$.

\begin{lem}
 $$K_n=\gpn{B_1,B_2,B_3,B_4,L,U_0,\ldots,U_{n-3},V_0,\ldots,V_{n-4},W_n,T_n}$$

where 
$$ B_1=a^2,B_2=b^2,B_3=c^2,B_4=bcd$$

$$L=b^2c^2d^2(bcd)^{-2} $$

$$ U_i=\sigma^i((ad)^4a^{-4}c^{-4}),V_i=\sigma^i((adacac)^4a^{-12}c^{-8}d^{-4})$$

$$ W_n=\sigma^{n-3}((ac)^4a^{-4}c^{-4}),T_n=\sigma^{n-3}((abac)^4a^{-8}b^{-4}c^{-4})$$
\end{lem}

\begin{proof}
 Let $K_n'$ be the subgroup in the lemma. Clearly $B_i,L \in K_n$. Also since $U_0,V_0 \in K_n$ and
$\sigma(K_n)\subset K_n$ we see that $U_i,V_i$ are elements of $K_n$. We have 

$$W_n=w_n \sigma^{n-3}(a^{-4}c^{-4}) \in K_n  $$
similarly $T_n \in K_n$.

The converse inclusion can be shown similarly using $\sigma(K_n')\subset K_n'$.
\end{proof}

Let $\ap$ denote equivalence modulo $[K_n,F]$.

\begin{lem}
In the group $K_n / [K_n,F]$ we have the equalities $$L^2\ap U_i^2\ap V_i^2 \ap  W_n^2\ap T_n^2\ap1 $$

\end{lem}

\begin{proof}

Observe that $x^2,[x,y] \in K_n$ where $x,y \in \{b,c,d\}^\pm$. Also
$$1\ap[x^2,y]=[x,y]^x[x,y]\ap[x,y]^2 $$

hence
$$L=b^2c^2d^2d^{-1}c^{-1}b^{-1}d^{-1}c^{-1}b^{-1}\ap dcbd^{-1}c^{-1}b^{-1}\ap dcb[d,c]c^{-1}d^{-1}b^{-1}$$
$$\ap [d,c]dcbc^{-1}[d,b]b^{-1}d^{-1}\ap [d,c][d,b]d[c,b]d^{-1}\ap[d,c][d,b][c,b] $$

therefore $L^2\ap 1$.

$$\begin{array}{lll}
 U_0^2 & =   & (ad)^4a^{-4}d^{-4}(ad)^4a^{-4}d^{-4} \\
       & \ap & a^{-1}d^{-1}a^{-1}d^{-1}a^{-1}d^{-1}a^{-1}d^{-1}a^{-1}aadadadad \\
       & =   & [a,(ad)^4] \ap 1
\end{array}$$
\hfill

$V_0$ is conjugate via $adac$ to the product:
$$\begin{array}{ll}
      acabacacabacacabacaca^{-12}c^{-8}d^{-4}
abacacabacacabacacabacaca^{-12}c^{-8}d^{-4}abac & \ap \\ 
a^{-1}c^{-1}a^{-1}d^{-1}a^{-1}c^{-1}a^{-1}c^{-1}a^{-1}d^{-1}a^{-1}c^{-1}a^{-1}c^{-1}a^{-1}d^{-1}a^{-1}c^{-1}a^{-1}c^{-1}a^{-1}& \\
d^{-1}a^{-1}c^{-1}a^{-1}aacabacaacabacaacabacaacabac  & \\
= [a,(acadac)^4]\ap 1 &       
\end{array}$$

Where the last equality is true since $(acadac)^4$ is a conjugate of $(adacac)^4 \in K_n$.\hfill

\hfill

For $W_n$, a similar calculation (like the one for $U_0$) gives the following:

$$W_n^2 \ap [\sigma^{n-3}(a),\sigma^{n-3}((ac)^4)]\ap 1 $$

Also similar computation (like the one for $V_0$) shows that $T_n^2$ is conjugate to

$$[\sigma^{n-3}(d),\sigma^{n-3}((baca)^4)]\ap 1$$
\end{proof}

\begin{lem}
 In $K_n / [K_n,F]$ we have $\gp{B_1,B_2,B_3,B_4}\cong C^4$ where $C$ denotes the infinite cyclic group.
\end{lem}

\begin{proof}
 We have the quotient map 
$$ K_n / [K_n,F] \longrightarrow K_n/(K_n \cap F')$$

and the right hand side is a free abelian group since
$$ K_n/(K_n \cap F') \cong K_nF'/F' \leq F/F'\cong C^4 $$

Now $B_1,B_2,B_3,B_4$ are mapped onto the vectors 
$$(2,0,0,0),(0,2,0,0),(0,0,2,0),(0,1,1,1)$$ 

respectively. Linear independence of these vectors proves the assertion.
\end{proof}

\begin{lem}
 We have the following isomorphism:
$$ K_n / [K_n,F] \cong  C^4 \times M_n$$

where $C^4$ is freely generated by $B_1,B_2,B_3,B_4$ and is isomorphic to $ K_n / (K_n \cap F')$

and $M_n$ is the torsion part generated by $\{L,U_i,V_i,W_n,T_n\}$, which is an elementary abelian 2-group isomorphic
to $(K_n\cap F') /[K_n,F] $.\label{ilkbol}
\label{basit}
\end{lem}

\begin{proof}
 We have the split exact sequence 
$$ 1 \rightarrow(K_n\cap F') /[K_n,F] \rightarrow K_n / [K_n,F] \rightarrow K_n / (K_n \cap F') \rightarrow 1$$

From previous lemma we have $ K_n / (K_n \cap F') \cong C^4$. Hence $$M_n\cong (K_n\cap F') /[K_n,F] $$ and is
elementary abelian 2-group by lemma 6.
\end{proof}

\begin{lem}
 The elements $L,U_0,W_3,T_3 $ are independent in $M_3$.
\end{lem}
\begin{proof}
 
Clearly $M_3$ maps to the abelian group 

$$Q_3=F'/([K_3,F]\gamma_5(F)F^{(2)}) $$ 

The result follows from the next lemma.
\end{proof}

\begin{lem}
$Q_3$ has the following presentation:\hfill

\hfill

\label{uuu}

Generators:

 \begin{itemize}
 \item $[a,b],[a,c],[a,d],[b,c]$
\item $[a,x,y],\quad  x\neq y, \quad x,y\in\{b,c,d\}$
\item $[a,x,y,z],\quad  x\neq y,\quad y \neq z$ and $(x,y,z)$ is not a permutation of $(b,c,d)$ where
$x,y \in \{b,c,d\}$ and $z \in\{a,b,c,d\}$
\end{itemize}

Relations:

\begin{itemize}
 \item commutativity relations
 \item $[a,b]^8=[a,c]^4=[a,d]^4=[b,c]^2=1$
 \item $[a,b,c]^4=[a,b,d]^4=[a,c,b]^2=[a,c,d]^2=[a,d,b]^2=[a,d,c]^2=1$
 \item $[a,x,y,z]^2=1$
\end{itemize}

Moreover the the images of $L,U_0,W_3,T_3$ in $Q_3$ are $[b,c],[a,d]^2,[a,c]^2,[a,b,c]^{-2}$ respectively.

\end{lem}

\begin{proof} (Throughout this proof $\ap$ denotes equivalence modulo $([K_3,F]\gamma_5(F)F^{(2)})$)\hfill

\hfill

 Since $a^2,b^2,c^2,d^2 \in K_3$, using standard commutator calculus and the fact that 
$\gamma_5(F)$ appears in the denominator of $Q_3$, it is easy to see that $Q_3$ is generated
by elements of the form

\begin{itemize}
 \item $[x,y],\quad x\neq y.\quad x,y \in \{a,b,c,d\}$
 \item $[x,y,z],\quad x\neq y , x,y,z \in \{a,b,c,d\}$
 \item $[x,y,z,w]. \quad x\neq y, \quad x,y,z,w \in \{a,b,c,d\}  $
\end{itemize}

Before beginning calculations, we wish to write two equalities which will be frequently used 
in the remainder:

\begin{equation}
 [x,yz]=[x,z][x,y][x,y,z]
\label{1}
\end{equation}

\begin{equation}
 [xy,z]=[x,z][x,y,z][y,z]
\label{2}
\end{equation}

Clearly, in $Q_3$ we have the following relations:

$$[x,B_i]=[x,L]=[x,U_0]=[x,W_3]=[x,T_3]=1 \: x\in\{a,b,c,d\} $$

Using these we will further reduce the system of generators.\hfill 

\hfill

Firstly, from equation (\ref{1}) we have:

\begin{equation*}
  [x,a]^2[x,a,a]=[x,a^2] \ap 1, \quad x\in\{b,c,d\}
\end{equation*}
Hence 
\begin{equation}
 [x,a,a]\ap[a,x]^2 
\label{aa}
\end{equation}

and we can omit the generators $[x,a,a]$ where $x\in\{b,c,d\}$.\hfill

Since 
\begin{equation}
 [x,a,y]=[a,x][y,[a,x]][x,a]\ap[y,[a,x]]=[a,x,y]^{-1} 
\label{3lu}
\end{equation}

we also can omit generators $[a,x,a]$ and $[x,a,y]$ where $x,y \in \{b,c,d\}$.\hfill

\hfill

Next, again using equation (\ref{1})  we have

\begin{equation*}
 [x,y]^2[x,y,y]=[x,y^2]\ap1\quad x\in \{a,b,c,d\} , y \in \{b.c.d\}
\end{equation*}

Hence
\begin{equation}
 [x,y,y]\ap [y,x]^2 
\label{yy}
\end{equation}

and therefore the generators $[a,y,y]$ where $y\in\{b,c,d\}$ can be omitted.\hfill

\hfill

Since $[x,y] \in K_3$ for $x,y \in \{b,c,d\}$, we also omit generators of the form $[x,y,a]$.\hfill

\hfill

Using equations (\ref{3lu}),(\ref{aa}) and (\ref{2}) we have 
$$[a,x,a,y]\ap[[x,a,a]^{-1},y]\ap[[a,x]^{-2},y]\ap[[a,x]^{-1},y]^2=[x,a,y]^2$$

which enables us to omit generators of the form $[a,x,a,y]$ where $x\in\{b,c,d\}$ and  $y\in\{a,b,c,d\}$.\hfill

\hfill

Similarly, using equations (\ref{yy}) and (\ref{2}) we have 
$$[a,x,x,z]\ap[[x,a]^2,z]\ap[[x,a],z]^2=[x,a,z]^2$$

which enables us to omit generators of the form $[a,x,x,z]$ where $x\in\{b,c,d\}$ and where $z\in\{a,b,c,d\}$.  \hfill

\hfill

The  following equation holds:

\begin{eqnarray}
 [x,yzt] &=&  [x,zt][x,y][x,y,zt]     \nonumber \\
   &=& [x,t][x,z][x,z,t][x,y][x,y,t][x,y,z][x,y,z,t]
\label{uzun}
\end{eqnarray}

Substituting $x=a$ yields the omission of the generators of the form $[a,y,z,t]$ where $(y,z,t)$ is a permutation
of $(b,c,d)$.\hfill

\hfill

Similarly, substituting different letters into equation (\ref{uzun}) we get the following identities:
$$1\ap[b,bcd]\ap[b,c][b,d]$$
$$1\ap[c,bcd]\ap[c,d][c,b]$$
$$1\ap[d,bcd]\ap[d,b][d,c]$$

Which yield the identities $[c,d]\ap[b,c]\ap[d,b]$. Thus $[c,d]$ and $[b,d]$ can be omitted from the system of generators.\hfill

\hfill

Hence $Q_3$ has the asserted set of generators. We proceed to showing it has the given relators.\hfill

\hfill

Using equation (\ref{yy}) we get $[b,c]^2\ap1$.\hfill

\hfill

Using similar calculations as before we get :

$$1\ap[a,(ad)^4]\ap[a,d]^4 $$
$$1\ap[a,(ac)^4]\ap[a,c]^4 $$
$$1\ap[a,(ab)^8]\ap[a,b]^8 $$

We have 
$$ 1\ap[x,(ad)^4]=[x,[a,d]^2]=[x,[a,d]]^2[[x,ad],[a,d]]\ap[x,[a,d]]^2\ap[a,d,x]^{-2}  $$
hence $[a,d,b]^2=[a,d,c]^2\ap1$\hfill

\hfill

Similarly 
$$1\ap[x,(ac)^4]\ap[a,c,x]^{-2} $$
$$1\ap[x,(ab)^8]\ap[a,b,x]^{-4} $$

yield $[a,c,b]^2\ap[a,c,d]^2\ap 1 $ and $[a,b,c]^4\ap[a,b,d]^4\ap1$\hfill

\hfill

Finally 
$$1\ap[a,x,y,z^2]\ap[a,x,y,z]^2[a,x,y,z,z]\ap[a,x,y,z]^2$$
where $x,y\in\{b,c,d\}$ and $a\in \{a,b,c,d\}$.\hfill

\hfill

Let us show that the images of $L.U_0,W_3,T_3$ in $Q_3$ are $[b,c],[a,d]^2,[a,c]^2,[a,b,c]^{-2}$ respectively:\hfill

\hfill

By Lemma 6, $L\ap[d,c][d,b][c,b]\ap[b,c]$\hfill

\hfill

Also similar to earlier computations we have;
$$U_0 \ap [a,d]^2 $$
$$W_3 \ap [a,c]^2 $$

$$
\begin{array}{ccl}
 T_3 & = & (abac)^4a^{-8}b^{-4}c^{-4}=abacabacabacabacabaca^{-8}b^{-4}c^{-4} \\
 & \ap &                             [a,b]b^{-1}c^{-1}a^{-1}bac[a,b]b^{-1}c^{-1}a^{-1}bac\\
  & =&                              ([a,b][b,ac])^2 \\
&= &                                ([a,b][b,c][b,a][b,a,c])^2 \\
&\ap&                                [b,a,c]^2\ap [a,b,c]^{-2} (\text{by equation  \ref{3lu}} )\\

\end{array}
$$

Therefore 
$$1\ap[a,T_3]\ap[a,[a,b,c]^{-2}]\ap[a,[a,b,c]]^{-2}\ap[a,b,c,a]^2 $$

Similarly 
$$1\ap[x,T_3]\ap[x,b,c,a]^2 $$

Now the following argument finishes the proof of the lemma:\hfill

\hfill

It is clear that all relations of $Q_3$ can be derived from the relations 

$$[x,B_i]=[x,L]=[x,U_0]=[x,W_3]=[x,T_3]=1, x\in\{a,b,c,d\} $$

together with relations of the form $y=1$ where $y\in \gamma_5(F)$ or $y\in F^{(2)}$.\hfill

\hfill

The relations $y\in \gamma_5(F)$ can be disregarded by omitting the generators of commutator length 5 or more
from the generating system. The relations $y\in F^{(2)}$ translate to commutativity relations among generators. 
Finally above computations show that one can further reduce the generating set to the one asserted in the lemma.
Also the calculations imply that the relations are equivalent to the system of relators given in the lemma.
Hence $Q_3$ has the given presentation, and clearly $L,U_0,W_3,T_3$ are independent in $Q_3$.

\end{proof}

\begin{lem}
 The elements $L,U_0,U_1,V_0 $ are independent in $M_n$ where $n\geq 4$.
\label{geyk}
\end{lem}

\begin{proof}
  $M_n$ maps to the abelian group $Q_n=F'/([K_n,F]\gamma_5(F)F^{(2)})$. 
The result is a corollary of the next lemma.
\end{proof}

\begin{lem}
 $Q_n$ has the following presentation:\hfill

\hfill

Generators:

 \begin{itemize}
 \item $[a,b],[a,c],[a,d],[b,c]$
\item $[a,x,y],\quad  x\neq y, \quad x,y\in\{b,c,d\}$
\item $[a,x,y,z],\quad  x\neq y,\quad y \neq z$ and $(x,y,z)$ is not a permutation of $(b,c,d)$ where
$x,y \in \{b,c,d\}$ and $z \in\{a,b,c,d\}$
\end{itemize}

Relations:

\begin{itemize}
 \item commutativity relations
 \item $[a,b]^{16}=[a,c]^8=[a,d]^4=[b,c]^2=1$
 \item $[a,b,c]^8=[a,b,d]^8=[a,c,b]^4=[a,c,d]^4=[a,d,b]^2=[a,d,c]^2=1$
 \item $[a,x,y,z]^2=1$
\end{itemize}

Moreover the the images of $L,U_0,U_1,V_0$ in $Q_n$ are $[b,c],[a,d]^2,[a,c]^4,[a,d]^2[a,c,d]^2$ respectively.
\end{lem}

\begin{proof} Most of the proof is similar to the proof of Lemma (\ref{uuu}). Additionally we only need to show that the relations:
$$[x,U_1]=\ldots=[x,U_{n-3}]=[x,V_0]=\dots=[x,V_{n-4}]=[x,W_n]=[x,T_n]=1$$

are consequences of the given system of relators.\hfill

\hfill

Let $\cong$ mean equality in $F'$ modulo the subgroup $[K_n,F]F^{(2)}$.\hfill

\hfill

Using equation (\ref{yy}) we have:
$$[a,c]^2\cong[a,c,c]^{-1} $$

Also using equation (\ref{1})
$$1\cong[a,c,c^2]\cong[a,c,c]^2[a,c,c,c] $$

which implies 
$$[a,c,c]^{-2}\cong[a,c,c,c] $$

and hence 
$$U_1\cong[a,c]^4\cong[a,c,c]^{-2}\cong[a,c,c,c] $$
which yields $U_1\in \gamma_4(F)$ mod $[K_n,F]F^{(2)}$.Therefore $[x,U_1]\in \gamma_5(F)$ mod $[K_n,F]F^{(2)}$.

It follows that relations of the form $[x,U_1]$ are consequences of previous relations. 
Since $\sigma(\gamma_5(F))\leq \gamma_5(F)$ and $\sigma(F^{(2)})\leq F^{(2)}$, we also see that relations of
the form $[x,U_i]$ where $i=2,\ldots,n-3$ are consequences of previous relations.\hfill

\hfill

For $V_0$ we have:
\begin{eqnarray*}
 V_0 & = & (adacac)^4a^{-12}c^{-8}d^{-4} \\
     & = & (ad)^2(ac)^2[(ac)^2,ad](ac)^2(ad)^2(ac)^2[(ac)^2,ad](ac)^2a^{-12}c^{-8}d^{-4}\\
    &  \ap & [a,d]^2[a,c]^4[[a,c],ad]^2 \\
    & \ap & [a,d]^2[a,c]^4[a,c,d]^2[a,c,a]^2[a,c,a,d]^2\\
    & \ap & [a,d]^2[a,c]^4[a,c,d]^2[c,a,a]^{-2}[a,c,a,d]^2\\
    & \ap &  [a,d]^2[a,c]^4[a,c,d]^2[a,c]^4[a,c,a,d]^2 \\
     & \ap & [a,d]^2[a,c,d]^2
\end{eqnarray*}

hence 
$$[x,V_0]\ap[x,[a,d]^2[a,c,d]^2]\ap[x,[a,d]]^2[a,[a,c,d]]^2=[a,d,x]^2[a,c,d,x]^2 $$

is a consequence of previous relations.\hfill

\hfill

For $V_1$ we get:
\begin{eqnarray*}
 V_1 & = & (acacacabacab)^4(aca)^{-12}b^{-8}c^{-4} \\
     & \cong & ((ac)^4ab[ab,ac][[ab,ac],ab])^4(aca)^{-12}b^{-8}c^{-4} \\
     & \cong & (ac)^{16}(ab)^8[[ab,ac],ab]^4(aca)^{-12}b^{-8}c^{-4} \\
     & \cong & [c,a,a,a,a][b,a,a,a]
\end{eqnarray*}

therefore $V_1 \in \gamma_4(F)$ mod $[K_n,F]F^{(2)}$ hence  relations of
the form $[x,V_i]$ where $i=1,\ldots,n-4$ are consequences of previous relations.
\begin{eqnarray*}
W_4 & =   & (acab)^4(aca)^{-4}b^{-4} \\
    & \cong & ([a,c][c,ab])^2 \\
    & \cong & [a,c]^2 ([c,b][c,a][c,a,b])^2\\
    & \cong & [c,a,b]^2
\end{eqnarray*}

But by equation (\ref{1}), 

$$1\cong[c,a,b^2]\cong[c,a,b]^2[c,a,b,b]$$

therefore 
$$W_4\cong[c,a,b,b]^{-1}$$
hence $W_n=\sigma^{n-4}(W_4)\in \gamma_4(F)$ modulo $[K_n,F]F^{(2)}$ and relations $[x,W_n]$ follow from 
the previous relations.\hfill

\hfill

Finally for $T_3$ from previous computations we get:

$$T_3\cong [a,b,c]^{-2} $$
using 
$$1\cong[[a,b],c^2]\cong[a,b,c]^2[a,b,c,c] $$

we get 
$$T_3\cong[a,b,c]^{-2}\cong[a,b,c,c] $$

hence $T_n=\sigma^{n-3}(T_3)\in \gamma_4(F)$ modulo $[K_n,F]F^{(2)}$ and relations $[x,T_n]$ follow from 
the previous relations.

\end{proof}

Let $P=\gp{b,c,d,b^a,c^a,d^a}_F$ be the lift of $St_\mathcal{G}(1)$ to $F$ and let  
 
$$\psi : P \rightarrow \Gamma \times \Gamma $$ be the homomorphism similar to $\bar\varphi$.

\begin{lem}
 We have $Ker(\psi)=\gpn{B_1,B_2,B_3,B_4,L,U_0,U_1,V_0}$
\label{kern}
\end{lem}
\begin{proof}
  Restatement of \cite{MR1676626} lemma 11.
\end{proof}

\begin{lem}
 For $n\geq 4$ the following isomorphism holds:

$$K_n / ([K_n, F]) \cong C^4 \times C_2^4 \times \psi(K_n) / \psi([K_n,F])   $$

where the factor $C^4$ is generated by $B_1,\ldots,B_4$ and $C_2^4$ is generated by $L,U_0,U_1,V_0$.
\label{ebe}
\end{lem}

\begin{proof}
 Let $$\psi_* : K_n/[K_n,F] \rightarrow \psi(K_n)/\psi([K_n,F]) $$

be the homomorphism induced by $\psi$.

Then by lemma (\ref{kern}) we have:

$$
\begin{array}{ccl}
Ker(\psi_*)  & = & (Ker(\psi))[K_n,F]/[K_n,F]\\
    & = & \gp{B_1,B_2,B_3,B_4,L,U_0,U_1,V_0}_{K_n/[K_n,F]} \\
  & \cong &  C^4 \times C_2^4
\end{array}
$$

Hence from lemma (\ref{ilkbol}) the result follows.
\end{proof}

\begin{lem}

 We have:
$$\psi(U_i)=(1,U_{i-1})$$
$$\psi(V_i)=(1,V_{i-1})$$
$$\psi(W_n)=(1,W_{n-1})$$
$$\psi(T_n)=(1,T_{n-1})$$

for $i\geq 1 $ and $n\geq 4$.

\label{parcala}
\end{lem}

\begin{proof}
We have 
$$ 
\psi\sigma \left\{
\begin{array}{c}
a \mapsto (d,a) \\
b \mapsto (1,b)\\
c \mapsto (a,c) \\
d \mapsto (a,d)
\end{array}
\right.
$$

Hence the image of $\pi_1\psi\sigma $ lies in the subgroup $\gp{a,d}_{\Gamma}$ and 
$\pi_2\psi\sigma$ is the identity map. Since $\gp{a,d}_{\Gamma}$ has exponent 4 we get 
the asserted equalities.
\end{proof}

Let $$\Theta_n=\gp{U_1,\ldots,U_{n-3},V_0,\ldots,V_{n-4},W_n,T_n}_{\Gamma}^\#$$
so that $\Gamma / \Theta_n \cong \mathcal{G}_n$.

\begin{lem}
 The following relations hold:

$$\psi(K_n)=\Theta_{n-1} \times \Theta_{n-1} $$

$$\psi([K_n,F])=([\Theta_{n-1},\Gamma]\times [\Theta_{n-1},\Gamma])  \Psi $$

where $\Psi \leq \Gamma \times \Gamma$ is the subgroup consisting of elements of the form $(w^{-1},w),w\in 
\Theta_{n-1}$.
\label{yazmadi}
\end{lem}

\begin{proof}
 Similar to \cite{MR1676626} lemma 14.
\end{proof}

\begin{lem}
 We have the isomorphism:

$$\psi(K_n) / \psi([K_n,F]) \cong \Theta_{n-1}/([\Theta_{n-1},\Gamma]) $$

and the generators $\psi(U_i),\psi(V_i),\psi(W_n),\psi(T_n)$ are mapped to the generators
$U_{i-1},V_{i-1},W_{n-1},T_{n-1}$ respectively.
\end{lem}

\begin{proof}
 By lemma \ref{yazmadi}, $\psi(K_n) / \psi([K_n,F])$ is isomorphic to 
\begin{equation}
(\Theta_{n-1} \times \Theta_{n-1})/(([\Theta_{n-1},\Gamma]\times [\Theta_{n-1},\Gamma])\Psi)  
\label{heyvan}
\end{equation}

Since $(1,x)^{-1}(x,1)=(x,x^{-1})\in \Xi$, (\ref{heyvan}) is generated by elements of the form $(1,x)$ where 
$x\in\{U_1,\ldots,U_{n-4},V_0,\ldots,V_{n-5},W_{n-1},T_{n-1}\}$.\hfill

\hfill

It is easy to check that the map 

$$(1,x) \mapsto  x $$

gives an isomorphism between (\ref{heyvan}) and $ \Theta_{n-1}/([\Theta_{n-1},\Gamma])$.
\end{proof}

\begin{lem}
$W_3,T_3$ are independent in  $\Theta_3 /([\Theta_3,\Gamma])$ 
\label{kucuk}
\end{lem}

\begin{proof}
 The proof is analogous to the proof of lemma (\ref{uuu}) and is omitted.
\end{proof}

\begin{lem}
For $n\geq 4$, $U_1,V_0$ are independent in $\Theta_n /([\Theta_n,\Gamma])$
\label{eyi}
\end{lem}

\begin{proof}
 The proof is analogous to the proof of lemma (\ref{geyk}) and is omitted.
\end{proof}

\begin{lem}
For $n \geq 4$ we have 

$$ \Theta_n /([\Theta_n,\Gamma]) \cong C_2^2 \times \Theta_{n-1} /([\Theta_{n-1},\Gamma]) $$

Where $U_1,V_0$ are generators of the factor $C_2^2$ and the images of elements 

$U_2,\ldots,U_{n-3},V_1,\ldots,V_{n-4},W_n,T_n$ are generators of the second factor. 
\label{yirmi}
\end{lem}

\begin{proof}
 Similar to the proof of lemma (\ref{ebe}) using lemma (\ref{eyi}).
\end{proof}

\textbf{Proof of Theorem 2.}
 We need to show $n\geq3$, we have $$H_2(\mathcal{G}_n,\field{Z}) \cong (C_2)^{2n-2} $$

 We claim that for $n\geq3$ 
$$\Theta_n /([\Theta_n,\Gamma]) \cong C_2^{2n-4} $$

The case $n=3$  follows from lemma (\ref{kucuk}). Assume it holds for $n>3$. Then by lemma (\ref{yirmi})

$$ \Theta_{n+1} /([\Theta_{n+1},\Gamma]) \cong C_2^2 \times \Theta_{n} /([\Theta_{n},\Gamma]) $$ 

and the claim follows from the induction hypothesis. Hence 

$$K_n/([K_n,F]) \cong C^4 \times C_2^{2n-2} $$

and the result follows from lemma (\ref{basit}) and Hopf's formula.\hfill

\hfill

\textbf{Proof of theorem 3.} Let $\bar K_n = K_n /[K_n,F]$ and 
$ \acute K = \bar K_n / \gp{B_1^2,B_2^2,B_3^2,B_4^2}$. We have the following homomorphism:

$$
\begin{array}{ccc}
\bar K_n & \longrightarrow & \acute K_n \\
a^2 & \mapsto & B_1 \\
b^2 & \mapsto & B_2 \\
c^2 & \mapsto & B_3 \\
d^2 & \mapsto & B_3B_2L \\
u_i & \mapsto & U_i \\
v_i & \mapsto & V_i \\
w_n & \mapsto & W_n \\
t_n & \mapsto & T_n
\end{array}
$$

Hence any dependence among the initial relators will produce dependence among generators of 
$\acute K_n$.

\subsection{Computation of the cohomology group  $H^2(\mathcal{\hat G}, \field{F}_2)$ }

This subsection is devoted to the proof of theorem 5.\hfill

\hfill

It is well known (see for example \cite{MR1200015}) that for a finite abelian group $A$ one has

$$H^2(G,A)\cong (G/G' \otimes A ) \times (M(G) \otimes A)$$

Using this we have :

\begin{lem}
  For $n\geq 3$   $H^2(\mathcal{G}_n,\field{F}_2) \cong C_2^{2n+1}$.
\label{koh}
\end{lem}

\begin{proof}
   As mentioned before we have $\mathcal{G}_n / \mathcal{G}_n' \cong (C_2)^3$.
Since $M(\mathcal{G}_n)\cong C_2^{2n-2}$ and $C_2 \otimes C_2 \cong C_2$ it follows that 
$$H^2(\mathcal{G}_n,\field{F}_2) \cong C_2^{2n+1} $$
\end{proof}
\begin{lem}
 For  natural numbers $n,k$ with $ n \geq 3$ let $$q_{n,k} :\mathcal{G}_{n+k} \rightarrow \mathcal{G}_n $$

be the canonical quotient map. Then there is $N\in \field{N}$ such that for all $n,k$
the dimension of the kernel of the induced map
$$q_{n,k}* :H^2(\mathcal{G}_n,\field{F}_2) \rightarrow  H^2(\mathcal{G}_{n+k},\field{F}_2) $$

is bounded above by $N$.
\label{bound}
\end{lem}

\begin{proof}

We have an exact sequence 

\begin{equation}
1 \rightarrow Ker(q_{n,k}) \rightarrow \mathcal{G}_{n+k} \rightarrow \mathcal{G}_n \rightarrow 1 
\label{ses}
\end{equation}

and clearly $Ker(q_{n,k})\cong St_\mathcal{G}(n)/St_\mathcal{G}(n+k)$.\hfill

\hfill

The sequence (\ref{ses}) induces the five term exact sequence (See \cite{MR1269324})
$$
\begin{array}{cccllll}
0 & \rightarrow & Hom(\mathcal{G}_n,\field{F}_2)& \xrightarrow{\alpha} & Hom(\mathcal{G}_{n+k},\field{F}_2) & \xrightarrow{\beta}

 & Hom(Ker(q_{n,k}),\field{F}_2)^{\mathcal{G}_n}  \\

& & &\xrightarrow{\partial} & H^2(\mathcal{G}_n,\field{F}_2) & \xrightarrow{q_{n,k}*} & H^2(\mathcal{G}_{n+k},\field{F}_2)
\end{array}
$$

where $Hom(Ker(q_{n,k}),\field{F}_2)^{\mathcal{G}_n}$ is the set of all homomorphisms invariant under the action of $\mathcal{G}_n$
on $Ker(q_{n,k})$ by conjugation.\hfill

\hfill

Since $\mathcal{G}_n$ and $\mathcal{G}_{n+k}$ have the same abelianization $\alpha$ is an isomorphism. Therefore $\beta$ is the zero map and 
hence $\partial$ is an injection.\hfill

\hfill

As noted earlier for $n\geq 4$ we have  $St_\mathcal{G}(n) \cong St_\mathcal{G}(n-1) \times St_\mathcal{G}(n-1) $.

Therefore we have $$Ker(q_{n,k})\cong \frac{St_\mathcal{G}(n)}{St_\mathcal{G}(n+k)}\cong \frac{St_\mathcal{G}(3)}{St_\mathcal{G}(k+3)} \times
\dots \times \frac{St_\mathcal{G}(3)}{St_\mathcal{G}(k+3)} $$

Since $\mathcal{G}_n$ acts transitively on these factors any homomorphism in
 $Hom(Ker(q_{n,k}),\field{F}_2)^{\mathcal{G}_n}$ 
is uniquely determined by its values in the first factor of this decomposition. But $St_\mathcal{G}(3)$ is finitely 
generated  hence the dimension of $Hom(Ker(q_{n,k}),\field{F}_2)^{\mathcal{G}_n}$ is no more than a fixed number and in 
particular independent of $k$ and $n$.
\end{proof}

\begin{lem}
 Suppose $\{G_i , \varphi_{ij} \mid i,j \in I\}$ is a direct system of  finitely generated elementary abelian p-groups.
 Suppose also  that the sequence  $dim(G_i)$ is monotone increasing and there is a uniform bound $N$ such that 
$\dim(Ker(\varphi_{ij}))\leq N$. Then the direct limit $\underrightarrow{\lim}\; G_i$ is infinite and hence 
isomorphic to $(C_p)^{\infty}$.
\label{sonradan}
\end{lem}
\begin{proof}
 Recall that the direct limit can be defined as the disjoint union $ \bigsqcup G_i$ factored by
the equivalence relation:
$$g_i \sim g_j \iff \exists k \geq i,j \text{  such that  } \varphi_{ik}(g_i)=\varphi_{jk}(g_j) $$
Suppose that $\underrightarrow{\lim}\; G_i$ has $M$ elements. Select $i$ large enough so that 
$dim(G_i) > NM$. So for large $j$ we have 

$$| G_i : Ker(\varphi_{ij}) | \geq \frac{dim(G_i)}{N} > \frac{NM}{N} =M $$

which shows that the direct limit has more than  $M$ elements.

\end{proof}

\textbf{Proof of theorem 5:} The following is well known: (See \cite{MR1691054} page 178) If $G$ is the inverse limit of the inverse system
$\{G_i , \phi_{ij}\}$ then 

$$H^n(G,\field{F}_p)\cong \underrightarrow{\lim} \: H^n(G_i,\field{F}_p) $$

i.e. $H^n(G,\field{F}_p)$ is the direct limit of the direct system $\{H^n(G_i,\field{F}_p),\phi_
{ij}^*\}$ where $\phi_{ij}^*$ is the inflation map induced by $\phi_{ij}$. Hence
$$H^2( \mathcal{\hat G}, \field{F}_2)\cong\lim_{\longrightarrow} H^2(\mathcal{G}_n,\field{F}_2) $$

Now lemmas  (\ref{koh}) and (\ref{bound}) show that the hypotheses of lemma (\ref{sonradan}) are satisfied
and therefore $H^2( \mathcal{\hat G}, \field{F}_2)\cong (C_2)^{\infty}$.\\

\textbf{Acknowledgements}\\

I would like to thank my advisors R. Grigorchuk and V. Nekrashevych for  valuable comments and helpful discussions. 

\bibliography{fquotient2.bib}
\end{document}